\documentclass[12pt,a4paper,reqno]{amsart}
\usepackage{amsmath, amsfonts, amsbsy, amsthm, amssymb}
\usepackage{graphicx}
\usepackage{xcolor}
\usepackage{enumerate,float,indentfirst}

\usepackage[T2A]{fontenc}
\usepackage[english]{babel}

\numberwithin{equation}{section}

\theoremstyle{plain}
\newtheorem{thm}{Theorem}[section]

\newtheorem{lem}[thm]{Lemma}

\theoremstyle{remark}

\theoremstyle{definition}
\newtheorem*{defn}{Definition}

\DeclareMathOperator{\E}{\mathbb{E}}

\DeclareMathOperator{\im}{Im}

\title[Local Laws for  Sparse Sample Covariance  Matrices]
      {Local Laws for  Sparse Sample Covariance Matrices\\
without the truncation condition}
      
\author[F.~G\"otze]{F.~G\"otze}

\address{Friedrich G{\"o}tze\\
 Faculty of Mathematics\\
 Bielefeld University \\
 Bielefeld, Germany
}

\email{goetze@math.uni-bielefeld.de}

\author[A.~Tikhomirov]{A.~Tikhomirov}

\address{Alexander N. Tikhomirov\\
 Institute of Physics and Mathematics\\
 Komi Science Center of Ural Division of RAS \\
 Syktyvkar, Russia; and National Research University Higher School of Economics, Moscow, Russia
 }

\email{tikhomirov@ipm.komisc.ru}

\author[D.~Timushev]{D.~Timushev}

\address{Dmitry A. Timushev\\
 Institute of Physics and Mathematics\\
 Komi Science Center of Ural Division of RAS \\
 Syktyvkar, Russia
 }

\email{timushev@ipm.komisc.ru}

     
\date{\today}

\begin{document}

\begin{abstract}
We consider sparse sample covariance matrices $\frac1{np_n}\mathbf X\mathbf X^*$, where $\mathbf X$ is a sparse matrix of order $n\times m$ with the sparse probability $p_n$. We prove the local Marchenko--Pastur law in some complex domain assuming that $np_n>\log^{\beta}n$, $\beta>0$ and some $(4+\delta)$-moment condition is fulfilled, $\delta>0$.
\end{abstract}

\maketitle

\section{Introduction}
Sample covariance matrices are of great practical importance for problems of multivariate statistical analysis and such rapidly developing areas as the theory of wireless communication and deep learning. Another significant area of application of sample covariance matrices is graph theory. The adjacency matrix of an undirected graph is asymmetric, so the study of its singular values leads to the sample covariance matrix. If we assume that the probability $p_n$ of having graph edges tends to zero as the number of vertices $n$ increases to infinity, we get to the concept of sparse random matrices.

Sparse Wigner random matrices have been considered in a number of papers (see \cite{Yau1, Yau2, Yau3, Yau4}) where many results have been obtained. With the symmetrization of sample covariance matrices it is possible to apply this results in the case when the observation matrix is square. However, when the sample size is greater than observation dimension, the spectral limit distribution has the singularity in zero, which requires different approaches. 

The limit spectral distribution of sparse sample covariance matrices with sparsity $np_n \sim n^\varepsilon$, ($\varepsilon>0$ is arbitrary small)  was studied in \cite{Lee1, Lee2}. In particular, a local law was proved under the assumption that the matrix elements satisfy the moments condition $ \E|X_{jk}|^q\le (Cq)^ {cq}$. In the paper \cite{mdpi} the case of the sparsity $np_n \sim \log^\alpha n$, for some $\alpha>1$ was considered, assuming that the moments of the matrix elements  satisfy the conditions $\E|X_{jk}|^{4+\delta}\le C<\infty$, $|X_{jk}|\le c_1(np_n)^{\frac12-\varkappa}$, for some $\varkappa>0$. Under this assumptions the local Marchenko--Pastur law was proved in some complex domain $z\in \mathcal D$ with $\im z > v_0 >0$, where $v_0$ is of order $\log^4 n/n$ and the domain bound not depend on $p_n$ while $np_n>\log^{\beta}n$.

This work is devoted to the case, when the elements $X_{jk}$ are not truncated, and only the conditions  $\E|X_{jk}|^{4+\delta}\le C<\infty$,  $np_n \sim \log^\alpha n$, for some $\alpha>1 $ are fulfilled.  We prove the  local Marchenko--Pastur law  in some complex domain $u+iv\in \mathcal D_\mu$ with the real part contained in the support of the Marchenko--Pastur distribution  and separated from the support ends.
\section{Main results}
Let $m=m(n)$, $m\ge n$.  Consider independent identically distributed zero mean random variables $X_{jk}$, $1\le  j\le n$, $1\le k\le m$ with $\E X^2_{jk}=1$ and independent of that set independent Bernoulli random variables $\xi_{jk}$, $1\le  j\le n$, $1\le k\le m$ with $\E\xi_{jk}=p_n$. In addition suppose that $np_n\to\infty$ as $n\to\infty$.

Observe the sequence of sparse sample covariance random matrices
\begin{equation*}
\mathbf X=\frac1{\sqrt{mp_n}}(\xi_{jk}X_{jk})_{1\le j\le n, 1\le k\le m}.
\end{equation*}
Denote by $s_1\ge \cdots\ge s_n$ the singular values of $\mathbf X$ and define the symmetrized empirical spectral distribution function (ESD) of the sample covariance matrix $\mathbf W=\mathbf X\mathbf X^*$:
$$
F_n(x)=\frac1{2n}\sum_{j=1}^n\Big(\mathbb I\{s_j\le x\}+\mathbb I\{-s_j\le x\}\Big),
$$
where $\mathbb I\{A\}$ stands for the event $A$ indicator.

Note that $F_n(x)$ is the ESD of the block matrix 
 $$
 \mathbf V=\begin{bmatrix}&\mathbf O_n&\mathbf X\\&\mathbf X^*&\mathbf O_m\end{bmatrix},
 $$
 where $\mathbf O_k$ is $k\times k$ matrix with zero elements. 

Denote $\mathbf R=\mathbf R(z)$ the resolvent matrix of $\mathbf V$:
$$
\mathbf R=(\mathbf V-z\mathbf I)^{-1}.
$$ 

Let $y=y(n)=\frac nm$ and $G_y(x)$ --- the symmetrized Marchenko--Pastur distribution function with the density
$$	
g_y(x)=\frac 1{2\pi y|x|}\sqrt{(x^2-a^2)(b^2-x^2)}\,\mathbb I\{a^2\le x^2\le b^2\},
$$
where $a=1-\sqrt y,\quad b=1+\sqrt y$. We shall assume that $y\le y_0<1$ for $n,m\ge1$.
Denote by $S_y(z)$ the Stieltjes transform of the distribution function $G_y(x)$ and $s_n(z)$ the Stieltjes transform of the distribution function $F_n(x)$. We have
\begin{align*}
S_y(z)=&\frac{-z+\frac{1-y}z+\sqrt{(z-\frac{1-y}z)^2-4y}}{2y},\\
s_n(z)=&\frac1{2n}\Big[\sum_{j=1}^n\frac1{s_j-z}+\sum_{j=1}^n\frac1{-s_j-z}\Big]=\frac1n\sum_{j=1}^n\frac z{s_j^2-z^2}=\frac1n\sum_{j=1}^nR_{jj}.
\end{align*}
The last equality follows from Schur complement (see \cite[Section 3]{mdpi}).
Put
\begin{equation}\label{b(z)1}
b(z)=z-\frac{1-y}z+2yS_y(z)=-\frac1{S_y(z)}+yS_y(z).
\end{equation}
In this paper we prove so called Marchenko--Pastur law for sparse sample covariance matrices.
Let 
$$
\Lambda_n:=\Lambda_n(z)=s_n(z)-S_y(z).
$$
For constant $\delta>0$ define the value $\varkappa=\varkappa(\delta):=\frac{\delta}{2(4+\delta)}$ and
consider the following conditions:
\begin{itemize}
\item the condition $(C0)$: for some $c_0>0$ and all $n\ge 1$ we have
$
np_n\ge c_0\log^{\frac2{\varkappa}} n;
$
\item the condition $(C1)$: for some $\delta>0$ we have 
$
\mu_{4+\delta}:=\E|X_{11}|^{4+\delta}<\infty;
$
\item the condition $(C2)$: there exists a constant $c_1>0$ such that for all $1\le j\le n$, $1\le k \le m$ we have
$|X_{jk}|\le c_1(np_n)^{\frac12-\varkappa}$ almost surely.
\end{itemize}

Introduce the quantity $v_0=v_0(a_0):=a_0n^{-1}\log^4 n$ with some positive constant $a_0$, and define the region
\begin{equation*}
\mathcal D(a_0):=\{z=u+iv: (1-\sqrt y-v)_+\le|u|\le1+\sqrt y + v, V\ge v\ge v_0\}.
\end{equation*}
Let
\begin{equation*}
\Gamma_n=2C_0\log n\Big(\frac1{nv}+\min\Big\{\frac1{np|b(z)|},\frac1{\sqrt{np}}\Big\}\Big),
\end{equation*}
\begin{equation*}
d(z)=\frac{\im b(z)}{|b(z)|},
\end{equation*}
and
\begin{equation*}
  d_n(z):=\frac1{nv}\left(d(z)+\frac{\log n}{nv|b(z)|}\right)+\frac1{np|b(z)|}.
\end{equation*}
Put
\begin{align*}
\mathcal T_n:=&\mathbb I\{|b(z)|\ge\Gamma_n\}\left( d_n(z)+{ d}_n^{\frac34}(z)\frac1{(nv)^{\frac14}}+{ d}_n^{\frac12}(z)\frac1{(nv)^{\frac12}}\right)\\&+\mathbb I\{|b(z)|\le\Gamma_n\}\left(\left(\frac{\Gamma_n}{nv}\right)^{\frac 12}+\Gamma_n^{\frac 12}\left(\frac{\Gamma_n^{\frac12}}{\sqrt{nv}}+\frac1{\sqrt{np}}\right)\right).
\end{align*}
In the paper \cite{mdpi}, assuming that the conditions $(C0)$--$(C2)$ are satisfied, the next theorem was proved:
\begin{thm}
Assume that
the conditions $(C0)$--$(C2)$ are satisfied. 
Then for any $Q\ge1$ there exist positive constants $C=C(Q, \delta, \mu_{4+\delta},c_0, c_1)$, $K=K(Q, \delta, \mu_{4+\delta},c_0, c_1)$, $a_0=a_0(Q, \delta, \mu_{4+\delta},c_0, c_1)$ such that for $z\in\mathcal D(a_0)$ 
\begin{equation*}
\Pr\Big\{\,|\Lambda_n|\ge K\mathcal T_n\Big\}\le Cn^{-Q}.
\end{equation*}
\end{thm}

This work is devoted to the case, when the elements $X_{jk}$ are not truncated, and only the conditions $(C0)$--$(C1)$ are fulfilled. 
Let
\begin{equation*}
\mathcal D_{\mu}=\{z=u+iv:\, 1-\sqrt y+\mu\le |u|\le1+\sqrt y-\mu,\, V\ge v\ge v_0\},
\end{equation*}
for some $\mu>0$. Note that $|b(z)|$ are bounded in domain $\mathcal D_\mu$, therefore
\begin{equation}\label{Gamma}
\Gamma_n=C_0\log n\Big(\frac1{nv}+\frac1{np}\Big).
\end{equation}

Without assumption $(C1)$ we get the following result.
\begin{thm}\label{main}
Assume that the conditions $(C0)$--$(C1)$ are satisfied. Then for any $\mu>0$ and $Q\ge1$ there exist constants $K=K(Q,\delta,\mu_{4+\delta},\mu)$, $a_0=a_0(Q,\delta,\mu_{4+\delta},\mu)$ depending on 
$Q$, $\delta$, $\mu_{4+\delta}$ and $\mu$ such that
$$
\Pr\{|\Lambda_n|\le K\Gamma_n\}\ge1-n^{-Q},
$$
for all $z\in \mathcal D_{\mu}$ and $\Gamma_n$ defined in \eqref{Gamma}.
\end{thm}

\subsection*{Organization}
The proof of the theorem is based on papers \cite{Aggar} and \cite{mdpi}.  In \textbf{Section \ref{sec4}} we follow \cite{mdpi}. In our case the domain $\mathcal D_\mu$ is separated from the ends of the spectrum. This makes it possible to significantly simplify the estimates obtained there and so to prove Theorem \ref{main}. In \textbf{Section \ref{sec2}}, we show that the elements $R_{jk}$ of the resolvent are bounded. For this, following \cite{Aggar}, we introduce the so-called admissible and inadmissible configurations. Assuming that the configuration is admissible, we obtain conditional estimates for $R_{jk}$. Further, taking into account the small probability of inadmissible configurations, we obtain the estimate for the resolvent elements. In the \textbf{Section \ref{aux}} we state and prove some auxiliary results.

\subsection*{Notation}
We use $C$ for large universal constants which maybe different from line by line. $S_y(z)$ and $s_n(z)$ denote the Stieltjes transforms of the symmetrized Marchenko--Pastur distribution and the spectral distribution function correspondingly. $R(z)$ denotes the resolvent matrix. Let $\mathbb T= \{1,\ldots,n\}$, $\mathbb J\subset\mathbb  T$ and
$\mathbb T^{(1)}=\{1,\ldots,m\}$, $\mathbb K\subset \mathbb T^{(1)}$. Consider $\sigma$-algebras $\mathfrak M^{(\mathbb J,\mathbb K)}$, generated by the elements of $\mathbf X$ with the exception of the rows with number from  $\mathbb J$ and the columns with number from $\mathbb K$. We will write for brevity $\mathfrak M_j^{(\mathbb J,\mathbb K)}$ 
instead of $\mathfrak M^{(\mathbb J\cup\{j\},\mathbb K)}$ and $\mathfrak M_{l+n}^{(\mathbb J,\mathbb K)}$ instead of $\mathfrak M^{(\mathbb J, \mathbb K\cup\{l\})}$. By symbol $\mathbf X^{(\mathbb J,\mathbb K)}$ we denote the matrix $\mathbf X$ which rows with numbers in $\mathbb J$ are deleted, and which columns with numbers in $\mathbb K$ are deleted too. In a similar way, we will denote all objects defined via $\mathbf X^{(\mathbb J,\mathbb K)}$, such that the resolvent matrix $\mathbf R^{(\mathbb J,\mathbb K)}$, the ESD Stieltjes transform $s_n^{(\mathbb J,\mathbb K)}$, $\Lambda_n^{(\mathbb J,\mathbb K)}$ and so on. 
 The symbol $\E_j$ denotes the conditional expectation with respect to the $\sigma$-algebra $\mathfrak M_j$, and $\E_{l+n}$ --- with respect to $\sigma$-algebra $\mathfrak M_{l+n}$.
Let ${\mathbb J}^c=\mathbb T\setminus\mathbb J$, ${\mathbb K}^c=\mathbb T^{(1)}\setminus\mathbb K$.

\section{Proof of Theorem \ref{main}}\label{sec4}
For the diagonal elements of $\mathbf R$ we can write
 \begin{equation}\label{000}
 R^{(\mathbb J,\mathbb K)}_{jj}=S_y(z)\big(1-\varepsilon^{(\mathbb J,\mathbb K)}_jR^{(\mathbb J,\mathbb K)}_{jj}+y\Lambda_n^{(\mathbb J,\mathbb K)}R^{(\mathbb J,\mathbb K)}_{jj}\big),
 \end{equation}
 for $j\in{\mathbb J}^c,$ and
 \begin{equation}\label{002}
 R_{l+n,l+n}^{(\mathbb J,\mathbb K)}=-\frac1{z+yS_y(z)}\big(1-\varepsilon^{(\mathbb J,\mathbb K)}_{l+n}R_{l+n,l+n}^{(\mathbb J,\mathbb K)}+y\Lambda_n^{(\mathbb J,\mathbb K)}R_{l+n,l+n}^{(\mathbb J,\mathbb K)}\big),
 \end{equation}
 for $l\in{\mathbb K}^c$. Correction terms $\varepsilon_{j}^{(\mathbb J,\mathbb K)}$ for $j\in{\mathbb J}^c$ and $\varepsilon_{l+n}^{(\mathbb J,\mathbb K)}$ for $l\in{\mathbb K}^c$ are defined as
 \begin{align*}
\varepsilon_j^{(\mathbb J,\mathbb K)}&=\varepsilon^{(\mathbb J,\mathbb K)}_{j1}+\cdots+\varepsilon^{(\mathbb J,\mathbb K)}_{j3},\notag\\
 \varepsilon^{(\mathbb J,\mathbb K)}_{j1}&=\frac1{m}\sum_{l=1}^mR_{l+n,l+n}^{(\mathbb J,\mathbb K)}-\frac1m\sum_{l=1}^mR_{l+n,l+n}^{(\mathbb J\cup \{j\},\mathbb K)},\notag\\
\varepsilon_{j2}^{(\mathbb J,\mathbb K)}&=\frac1{mp}\sum_{l=1}^m(X_{jl}^2\xi_{jl}-p)R^{(\mathbb J\cup \{j\},\mathbb K)}_{l+n,l+n},\notag\\
\varepsilon_{j3}^{(\mathbb J,\mathbb K)}&=\frac1{mp}\sum_{1\le l\ne k\le m}X_{jl}X_{jk}\xi_{jl}\xi_{jk}R^{(\mathbb J\cup \{j\},\mathbb K)}_{l+n,k+n};
\end{align*}
and
\begin{align*}
\varepsilon_{l+n}^{(\mathbb J,\mathbb K)}&=\varepsilon_{l+n,1}^{(\mathbb J,\mathbb K)}+\cdots+\varepsilon_{l+n,3}^{(\mathbb J,\mathbb K)},\notag\\
\varepsilon_{l+n,1}^{(\mathbb J,\mathbb K)}&=\frac1m\sum_{j=1}^nR^{(\mathbb J,\mathbb K)}_{jj}-\frac1m\sum_{j=1}^nR^{(\mathbb J,\mathbb K\cup\{l+n\})}_{jj},\notag\\
\varepsilon_{l+n,2}^{(\mathbb J,\mathbb K)}&=\frac1{mp}\sum_{j=1}^n(X_{jl}^2\xi_{jl}-p)R^{(\mathbb J,\mathbb K\cup\{l+n\})}_{jj},\notag\\
\varepsilon_{l+n,3}^{(\mathbb J,\mathbb K)}&=\frac1{mp}\sum_{1\le j\ne k\le n}X_{jl}X_{kl}\xi_{jl}\xi_{kl}R^{(\mathbb J,\mathbb K\cup\{l+n\})}_{jk}. 
\end{align*}

Summing the equation \eqref{000} ($\mathbb J=\emptyset$, $\mathbb K=\emptyset$), we get the self-consistent equation
 \begin{equation*}
 s_n(z)=S_y(z)(1+T_n-y\Lambda_ns_n(z)),
 \end{equation*}
with the error term 
\begin{equation*}
T_n=\frac1n\sum_{j=1}^n\varepsilon_jR_{jj}.
\end{equation*}

The proof of Theorem \ref{main} is based on the following theorem.
\begin{thm}\label{thmT}
Under the conditions of the Theorem \ref{main}, for any $\mu>0$, there exist constants $C=C(\delta, \mu_{4+\delta},c_0)$, $a_0=a_0(\delta, \mu_{4+\delta},c_0)$, such that
\begin{equation*}
\E|T_n|^q\mathbb I\{\mathcal Q\}\le C^q\Big(\frac1{nv}+\frac 1{np}\Big)^q\log^qn,
\end{equation*}
for all $z\in \mathcal D_{\mu}$.
\end{thm}
\begin{proof}
The proof repeats \cite{mdpi}[Theorem 3], taking into account that $0<\varepsilon<\im b(z)$ for some $\varepsilon>0$ and $\im b(z)$, $|b(z)|$ are bounded in domain $\mathcal D_\mu$. The arguments of \cite{mdpi}[Theorem 3] also require that the condition $\Pr\{\mathcal B\}\le Cn^{-Q}$ be satisfied (see \cite{mdpi}[p. 17]). But Lemma \ref{lem_main} implies $\Pr\{\mathcal B; \mathcal Q\}\le Cn^{-Q}$.
\end{proof}
\begin{proof}[Proof of Theorem \ref{main}]
First of all, we note that \cite[Lemma 8]{mdpi} gives the bound
	\begin{equation*}
	|\Lambda_n|\le C{|T_n|}
	\end{equation*}
	in domain $\mathcal D_\mu$. We have
	$$
\Pr\{|\Lambda_n|\ge K\Gamma_n\}\le\Pr\{|\Lambda_n|\ge K\Gamma_n; \mathcal Q\} +\Pr\{\mathcal Q^c\}.
$$
\cite[Corollary 3]{mdpi} implies 
\begin{equation*}
\Pr\{\mathcal Q\}\ge1-Cn^{-Q}.
\end{equation*}
Applying Markov inequality and combining the last inequality and Theorem \ref{thmT}, we get
\begin{equation*}
\Pr\{|\Lambda_n|\ge K\Gamma_n\}\le\frac{\E|T_n|^q\mathbb I\{\mathcal Q\}}{{K^q}\Gamma_n^q}+Cn^{-Q}\le\Big(\frac{C}{K}\Big)^q.
\end{equation*}

By choosing a sufficiently large K value and  $q\sim\log n$, we obtained the proof.
\end{proof}

\section{Estimate of $R_{jk}$}\label{sec2}

 We shall use the notations of \cite{mdpi}.

Let $s_0>1$ be some positive constant depending on $\delta$, $V$. 
For any $0<v\le V$ we define $k_v$ as
$$
k_v=k_v(V):=\min\{l\ge 0: s_0^lv\ge V\}.
$$
For given $\gamma>0$ consider the event
$$
\mathcal Q_{\gamma}(v):=\big\{|\Lambda_n(u+iv)|\le \gamma, \text{ for all } u\big\}
$$
and the event
\begin{equation*}
\mathcal Q:=\widehat{\mathcal Q}_{\gamma}(v)=\bigcap_{l=0}^{k_v}\mathcal Q_{\gamma}(s_0^lv).
\end{equation*}

For the proof of main result it is enough to estimate the entries of the resolvent matrix. We prove the next Lemma.
\begin{lem}\label{lem_main}
Under conditions of Theorem \ref{main} there exists a constant $H$ such that for $z\in\mathcal D_{\mu}$
\begin{equation*}
\Pr\{\max_{1\le j,k\le n+ m}|R_{jk}|>H;\mathcal Q\}\le Cn^{-c\log n\log n}.
\end{equation*}
\end{lem}
	Following the work of Aggarwal (see \cite{Aggar}), we introduce the configuration matrix $\mathbf L=(L_{jk})$.
	Set events
	$$
	A_{jk}=\{|X_{jk}|\ge C(np)^{\frac12-\varkappa}\}.
	$$
	Define the matrix $\mathbf L$ with elements
	$$
	L_{jk}=\xi_{jk}\mathbb I\{A_{jk}\}.
	$$
Note that
\begin{equation*}
\E L_{jk}\le \frac{\mu_{4+\delta}}{n^2p}.
\end{equation*} 
Introduce the configuration matrix $\mathbf L_{\mathbf V}$:
$$
\mathbf L_{\mathbf V}=\begin{bmatrix}\mathbf O&\mathbf L\\\mathbf L^T&\mathbf O\end{bmatrix}.
$$
\begin{defn}
We call $j$ and $k$ \emph{linked} (with respect to $\mathbf L_{\mathbf V}$), if $L_{jk}=1$. Otherwise we call them \emph{unlinked}.
\end{defn}
\begin{defn}
If there exists a sequence $j=j_1,j_2,\ldots,j_r=k$ such that $j_\nu$ is linked to $j_{\nu+1}$ for each $\nu\in[1,r-1]$, then $j$ and $k$ are called \emph{connected}.
\end{defn}
\begin{defn}
We call an index $j$ \emph{deviant} if there exists some index $k$ such that $j$ and $k$ are linked.
Otherwise we call $j$ \emph{typical}. 
\end{defn}
Let
$$
\mathcal D_{\mathbf L}=\{j\in[1,n+m]: j\text{ is deviant} \},\quad \mathcal T_{\mathbf L}=\{j\in[1,n+m]: j\text{ is typical}\}.
$$
\begin{defn}
We call $\mathbf L_{\mathbf V}$ \emph{deviant-inadmissible} if there exist at least $\sqrt{\frac{n}{p}}$, deviant indices. We call $\mathbf L_{\mathbf V}$ \emph{connected-inadmissible} if there exist distinct indices $j_1,j_2,\ldots, j_r$, $r=[\log n]$, that are pairwise connected. We call the configuration $\mathbf L_{\mathbf V}$ \emph{inadmissible}, if it is either deviant-inadmissible or connected-inadmissible. Otherwise, the configuration is called admissible.
\end{defn}
Define $\mathcal A$ as the set of all admissible configurations of size $n+m$.
Let $\mathcal C=\mathcal C_1\cup\mathcal C_2$ be the event that the configuration $\mathbf L_{\mathbf V}$ is inadmissible, $\mathcal C_1$ be the event that the configuration $\mathbf L_{\mathbf V}$ is deviant-inadmissible, and $\mathcal C_2$ be the event that the configuration $\mathbf L_{\mathbf V}$ is connected-inadmissible.
\begin{lem}\label{inadm}
Under the conditions of Theorem \ref{main} the bound
\begin{equation*}
\Pr\{\mathcal C\}\le Cn^{-c\log\log n}
\end{equation*}
is valid.
\end{lem}
\begin{proof}First, we estimate $\Pr\{\mathcal C_1\}$.
	The event $\mathcal C_1$ implies that there are at least $\sqrt{\frac{n}{p}}$ deviant indices, which in turn gives that there is at least $\sqrt{\frac{n}{p}}$ pairs $\{j,k\}$ such that $j\in[1,n]$, $k\in[1,m]$ and $L_{jk}=1$. Hence
	\begin{equation*}
	\Pr\{\mathcal C_1\}\le \sum_{j=\sqrt{\frac{n}{p}}}^{n}\binom{nm}{j}\bigg(\frac C{n^2p}\bigg)^j.
	\end{equation*}
By Stirling's formula, we have
\begin{equation*}
\binom{nm}{j}\left(\frac C{n^2p}\right)^j\le C \left(\frac C{\sqrt{np}}\right)^{j}
\end{equation*}
for $\sqrt{\frac{n}{p}}\le j\le n$.
This yields
\begin{equation*}
\Pr\{\mathcal C_1\}\le C\left(\frac{C}{\sqrt{np}
}\right)^{\sqrt{\frac np }}.
\end{equation*}

The estimate $\Pr\{\mathcal C_2\}$ almost repeats the proof of the bound for $\Pr\{\Delta_2\}$ in Lemma 3.11 of \cite{Aggar}. The event $\mathcal C_2$ implies that there exists a sequence of indices $\mathcal S=\{i_1,i_2,\ldots,i_r\}$ such that at least $r-1$ pair $(i_j,i_k)$ are linked. 
We have
\begin{equation*}
\Pr\{\mathcal C_2\}\le \binom{n+m}{r}\binom{r^2}{r-1}\left(\frac C{n^2p}\right)^{r-1}.
\end{equation*}
Applying Stirling's formula, get
\begin{equation*}
\Pr\{\mathcal C_2\}\le n^{-C\log\log n}.
\end{equation*}
\end{proof}
Now we fix the admissible configuration $\mathbf L_{\mathbf V}$. Let $R<\sqrt{\frac{n}{p}}$ denotes the number of the deviant indices. Consider the matrix
$\mathbf V_{\mathbf L}=(V_{\mathbf L}(j,k))$ with entries
$$
V_{\mathbf L}(j,k)=\begin{cases}0,\text{ if }1\le j,k\le n\text{ or } n+1\le j,k\le n+m,\\
\xi_{jk}a_{jk},\text{ if } 1\le j\le n, n+1\le k\le n+m\text{ and }L_{jk}=0,\\
\xi_{jk}b_{jk},\text{ if } 1\le j\le n, n+1\le k\le n+m\text{ and }L_{jk}=1,\\
\overline V_{kj}, \text{ if } n+1\le j\le n+m,\, 1\le k\le n.
\end{cases}
$$
Here $a_{jk}$ (resp. $b_{jk}$ ) are independent random variables with the distributions
\begin{equation*}
\Pr \{a_{jk}\in G\}=\Pr\{X_{jk}\in G\big|\mathcal A_{jk}^c\}
\end{equation*}
and
\begin{equation*}
\Pr \{b_{jk}\in G\}=\Pr\{X_{jk}\in G\big|\mathcal A_{jk}\}.
\end{equation*}
The permutation of rows and columns gives the matrix
\begin{equation*}
\mathbf V=\begin{bmatrix}&\mathbf V_{11}&\mathbf V_{12}\\&\mathbf V_{12}^*&\mathbf V_{22}\end{bmatrix}.
\end{equation*} 
The Hermitian matrix $\mathbf V_{11}$ of size $R\times R$ consists of type $b$ elements and has the form
\begin{equation*}
\mathbf V_{11}=
\begin{bmatrix}
&{\mathbf B_1} &0 \ldots &0 \ldots\\
&0 \ldots &\mathbf B_2& \ldots\\
&\ldots &\ldots &\ldots\\
&0 \ldots &0\ldots &\mathbf B_L
\end{bmatrix},
\end{equation*}
where  $\mathbf B_\nu$ are Hermitian matrices of order $r_\nu\le r$, $\nu=1,\ldots,L$.
The matrix $\mathbf V_{12}$ of size $R\times(m+n-R)$ consists of type $a$ elements and has the form
\begin{equation*}
\mathbf V_{12}=
\begin{bmatrix}
{\mathbf O_1} &{\mathbf A_1}
\end{bmatrix},
\end{equation*}
where $\mathbf O_{1}$ is a matrix of size $R\times m$ with zero elements, the matrix $\mathbf A_{1}$ is $R\times (n-R)$  with elements distributed by type $a$. 
The Hermitian matrix $\mathbf V_{22}$ of size $(n+m-R)\times (n+m-R)$ has the form
\begin{equation*}
\mathbf V_{22}=\begin{bmatrix}&\mathbf O_{11}&\mathbf A_{2}\\&\mathbf A_{2}^*&\mathbf O_{22}\end{bmatrix}.
\end{equation*} 
Here the square matrices $\mathbf O_{11}$ and $\mathbf O_{22}$ have zero elements and the orders $m$ and $n-R$ respectively, and the matrix $\mathbf A_{2}$ is $m\times (n-R)$  with elements distributed by type $a$.
The resolvent  $\mathbf R(z)=(\mathbf V-z\mathbf I)^{-1}$ can be represented as
\begin{equation*}
\mathbf R=\begin{bmatrix}
&\mathbf R_{11}&\mathbf R_{12}\\&\mathbf R_{12}^T&\mathbf R_{22}
\end{bmatrix},
\end{equation*}
where
\begin{align*}
\mathbf R_{11}=&(\mathbf V_{11}-z\mathbf I-\mathbf V_{12}(\mathbf V_{22}-z\mathbf I)^{-1}\mathbf V_{12}^*)^{-1},\\
\mathbf R_{12}=&(\mathbf V_{12}(\mathbf V_{22}-z\mathbf I)^{-1}\mathbf V_{12}^*-\mathbf V_{11}+z\mathbf I)^{-1}\mathbf V_{12}(\mathbf V_{22}-z\mathbf I)^{-1},\\
\mathbf R_{22}=&(\mathbf V_{22}-z\mathbf I)^{-1}+(\mathbf V_{22}-z\mathbf I)^{-1}\mathbf V_{12}^*\notag\\&\qquad\times(\mathbf V_{11}-z\mathbf I-\mathbf V_{12}(\mathbf V_{22}-z\mathbf I)^{-1}\mathbf V_{12}^*)^{-1}\mathbf V_{12}(\mathbf V_{22}-z\mathbf I)^{-1}.
\end{align*}
We will be primarily interested in estimating the spectral norm of the matrix $\mathbf R_{11} $ since it majorizes all elements of the matrix $\mathbf R_{11} $. Note that the dimension of the matrix $\mathbf R_{11}$ is equal to $R\times R$, where $R<\sqrt{\frac np}$.
Introduce a random matrix
$$
\mathbf Y=\mathbf V_{12}(\mathbf V_{22}-z\mathbf I)^{-1}\mathbf V_{12}^*.
$$
Note that
\begin{equation*}
\mathbf R^{(\mathbb J)}=(\mathbf V_{22}-z\mathbf I)^{-1}=\begin{bmatrix}
&\mathbf R^{(\mathbb J)}_{11}&\mathbf R^{(\mathbb J)}_{12}\\&{\mathbf R^{(\mathbb J)}_{12}}^T&\mathbf R^{(\mathbb J)}_{22}
\end{bmatrix}.
\end{equation*}
Given the form of the matrices $\mathbf V_{12}$ and $\mathbf V_{22}$, we find that 
$$
\mathbf Y=\mathbf A_{1}\mathbf R^{(\mathbb J)}_{22}\mathbf A_{1}^*.
$$
In these notation
$$
\mathbf R_{11}=(\mathbf V_{11}-z\mathbf I-\mathbf Y)^{-1}.
$$

In what follows we shall assume that $\mathbf L_\mathbf V$ is admissible. We prove that for the resolvent matrix $\mathbf R$ all entries are bounded conditioning by admissible $\mathbf L_\mathbf V$.
\begin{lem}\label{typical}
Let $\mathbf L_\mathbf V$ be admissible. Under conditions of Theorem \ref{main} there exists a constant $H$ such that for $z\in\mathcal D_{\mu}$
\begin{equation*}
\Pr\{\max_{1\le j,k\le n+ m}|R_{jk}|>H;\mathcal Q\}\le Cn^{-c\log\log n}.
\end{equation*}
\end{lem}
 Note that $\mathcal T_{\mathbf L}\cup\mathcal D_{\mathbf L}=[1,n+m]$, $\mathbb J \subset [1,n+m]$. We introduce the events
$$
\mathcal C_1(v,k)=\bigcap_{|\mathbb J|\le k}\bigg\{\max_{j,l\in\mathcal T_{\mathbf L}}|R^{(\mathbb J)}_{jl}(u+iv)|\le H_1\bigg\}
$$
and
$$
\mathcal C_2(v,k)=\bigcap_{|\mathbb J|\le k}\bigg\{\max_{j\in\mathcal D_{\mathbf L},1\le l\le n+m}|R^{(\mathbb J)}_{jl}(u+iv)|\le {H_2}\bigg\}.
$$

The following lemma holds.
\begin{lem}\label{lem_2.3}Under the conditions of the Theorem \ref{main}, the inequalities
	\begin{equation}\label{typic}
	\Pr\Big\{\mathcal C_1(v,k-1);\mathcal C_1(sv,k)\cap\mathcal C_2(sv,k)\cap\mathcal Q\Big\}\ge 1-Cn^{-c\log\log n}
	\end{equation}
	and
	\begin{equation}\label{devia}
	\Pr\Big\{\mathcal C_2(v,k-1);\mathcal C_1(sv,k)\cap\mathcal C_2(sv,k)\cap\mathcal Q\Big\}\ge 1-Cn^{-c\log\log n}
	\end{equation}
	are valid.
\end{lem}
\begin{proof}For simplicity, we assume that $k=1$.
	We begin by proving the inequality \eqref{typic}. Since both indices are typical, the corresponding matrix elements in the rows (and columns) with numbers $j,k$ are of type $a$.
	Consider the diagonal elements.
	For $j\in\mathcal T_{\mathbf L}\cap[1,n]$ the equality
	\begin{equation*}
	R_{jj}=yS_y(z)\Big(1+\varepsilon_jR_{jj}+\Lambda_nR_{jj}\Big)
	\end{equation*}
	holds.
	For $\omega\in\mathcal Q$ we have
	$$
	|\Lambda_n|\le \frac12.
	$$
	Hence,
	\begin{equation*}
	|R_{jj}|\mathbb I\{\mathcal Q\}\le 2\sqrt y(1+|\varepsilon_j||R_{jj}|)\mathbb I\{\mathcal Q\}.
	\end{equation*}
	Let
	\begin{equation*}
	\varepsilon_j=\varepsilon_{j1}+\varepsilon_{j2}+\varepsilon_{j3}
	\end{equation*}
	with
	\begin{align*}
	\varepsilon_{j1}=&\frac1{m}\sum_{l=1}^mR^{(j)}_{l+n,l+n}-\frac1{m}\sum_{l=1}^mR_{l+n,l+n},\notag\\
	\varepsilon_{j2}=&\frac1{mp}\sum_{l=1}^m(a^2_{jl}\xi_{jl}-p)R^{(j)}_{l+n,l+n},\notag\\
	\varepsilon_{j3}=&\frac1{mp}\sum_{l,t=1}^ma_{jl}a_{jt}\xi_{jl}\xi_{jt}R^{(j)}_{l+n,t+n}.
	\end{align*}
	Note that for admissible configurations
	\begin{equation*}
	|\mathcal D_{\mathbf L}|\le \sqrt{\frac np}.
	\end{equation*}
	By \cite[Lemma 1]{mdpi},
	\begin{equation*}
	|\varepsilon_{j1}|\le \frac C{nv}.
	\end{equation*}
	Next, note that
	\begin{equation*}
	\frac1n\sum_{l=1}^m|R^{(j)}_{l+n,l+n}|^2\mathbb I\{\mathcal C_1(sv,1)\}\mathbb I\{\mathcal C_2(sv,1)\}\mathbb I\{\mathcal Q\}\le
	{H_2^2s^2}+H_1^2s^2
	\end{equation*}
	and
	\begin{equation*}
	\frac1n\sum_{l=1}^m |R^{(j)}_{l+n,l+n}|^q\mathbb I\{\mathcal C_1(sv,1)\}\mathbb I\{\mathcal C_2(sv,1)\}\mathbb I\{\mathcal Q\}\le 
	{H_2^qs^q}+H_1^qs^q.
	\end{equation*}
	We used here the so-called multiplicative inequality: for any $s\ge 1$
	\begin{equation*}
	|R_{jj}(u+iv)|\le s|R_{jj}(u+isv)|.
	\end{equation*}
	Given the above, get
	\begin{align*}
	\E|\varepsilon_{j2}|^q\mathbb I\{\mathcal C_1(sv,1)\}\mathbb I\{\mathcal C_2(sv,1)\}\mathbb I\{\mathcal Q\}&\le \frac {C^qq^{\frac q2}s^{q}H_2^{q}}{(np)^{\frac q2}}+\frac {C^qs^{q}q^{\frac q2}H_1^q}{(np)^{\frac q2}}\notag\\&+\frac{C^qq^qH_2^qs^{q}}{(np)^{2\varkappa q+1}}+\frac{C^qq^qH_1^{q}s^{q}}{(np)^{2\varkappa q+1}}.
	\end{align*}
	Similarly,
	\begin{align*}
	\E|\varepsilon_{j3}|^q|R_{jj}|^q\mathbb I\{\mathcal C_1(sv,1)\}&\mathbb I\{\mathcal C_2(sv,1)\}\mathbb I\{\mathcal Q\}\le
	\frac{C^qq^q}{(nv)^q}a_n^{\frac q2}(z)
	+\frac{C^qq^{\frac {3q}2}s^{\frac {q}2}H_2^{\frac q2}}{(nv)^{\frac q2}(np)^{\varkappa q+1}}
	\notag\\&
	+\frac{C^qq^{\frac {3q}2}s^{\frac {3q}2}H_1^{\frac q2}}{(nv)^{\frac q2}(np)^{\varkappa q+1}}+
	\frac{C^qq^{2q}s^{2q}H_2^{2q}}{(np)^{2\varkappa q+2}}+\frac{C^qq^{2q}s^{2q}H_1^{2q}p}{(np)^{2\varkappa q+2}}.
	\end{align*}
	Here we used the fact that 
	\begin{align*}
	|R_{jk}(u+iv)|\le |R_{jk}(u+iv)|+(s-1)v|[\mathbf R(u+iv)\mathbf R(u+sv)]_{jk}|\le |R_{jk}(u+isv)|\notag\\+
	(s-1)\sqrt{\im R_{jj}\im R_{kk}}\le sH_1
	\end{align*}
	for $j,k\in \mathcal T_{\mathbf L}$, and
	\begin{equation*}
	|R_{jk}(u+iv)|\le {sH_2}
	\end{equation*}
in the case $j\in\mathcal D_{\mathbf L}$ or $k\in\mathbf D_{\mathbf L}$.

If $(np)^{{2\varkappa}}|b(z)|\ge C q^2sH_1$ and $(np)^{\varkappa}>CqsH_1$, then $H_1$ and $H_2$ can be chosen so that
\begin{equation*}
\E|R_{jj}|^q\mathbb I\{\mathcal C_1(sv,1)\}\mathbb I\{\mathcal C_2(sv,1)\}\mathbb I\{\mathcal Q\}\le H_1^p.
\end{equation*}
Now consider the case of deviant indices. Let $j\in\mathcal D_{\mathbf L}$ and $k$  be arbitrary.
Consider the matrix
\begin{equation*}
\mathbf Y=\mathbf V_{12}(\mathbf V_{22}-z\mathbf I)^{-1}\mathbf V_{12}^*=\mathbf A_{1}\mathbf R^{(\mathbb J)}_{22}\mathbf A_{1}^*.
\end{equation*}
We estimate the matrix $\mathbf Y$ elementwise. We start with off-diagonal elements. Consider $Y_{12}$.
The equality
$$
Y_{12}=\frac1{mp}\sum_{l,t}a_{1l}\overline a_{2t}[\mathbf R_{22}^{(\mathbb J)}]_{lt}
$$
holds.
Note that $\{a_{1l}\}$ and $\{a_{2t}\}$ are independent. We can apply the lemma \ref{bilin} with $\mathbf A=\mathbf R_{22}^{(\mathbb J)}$. 
By the assumption $\mathcal C_1\cap\mathcal C_2\cap{\mathcal Q}$ we get
\begin{equation*}
\|\mathbf A\|^2\le \gamma\frac{na_n(z)}{v}+\frac{r}{v},
\end{equation*}
and
\begin{equation*}
\sum_{j=1}^n\mathcal L_j^q\le \frac{H_2^{\frac q2}s^{\frac q2}}{pv^{\frac q2}|b(z)|^{\frac q2}}+\frac{H_1^{\frac q2}s^{\frac q2}n}{v^{\frac q2}}.
\end{equation*}
Finally,
\begin{equation*}
\sum_{i,j\in\mathbb T\setminus \mathbb J}|[\mathbf R^{(\mathbb J)}_{22}]_{ij}|^q\le \frac{H_2^{q-2}s^{q-2}n}{|b(z)|^{q-2}v}(\gamma a_n(z)+\frac{r}{nv}).
\end{equation*}
Further, we have
\begin{equation*}
\mu_{\xi}^{(q)},\mu_{\eta}^{(q)}\le p(np)^{-2-\varkappa q}
\end{equation*}
for $q\ge 4+\delta$, and
\begin{equation*}
\mu_{\xi}^{(q)},\mu_{\eta}^{(q)}\le p\mu_{4+\delta}^{\frac q{4+\delta}}/(np)^{\frac q2}
\end{equation*}
for $q\le 4$.
Combining all the estimates, we obtain
\begin{align*}
\mathcal A_1\le&
\frac{C^qq^{\frac q2}}{n^{q}}\Big(\frac1{q^{\frac q2}}+\frac{q^{\frac q2}}{(np)^{\frac q2\varkappa}}+\frac1{(np)^2}\frac{q^q}{(np)^{q\varkappa}}\Big),\\
\mathcal A_2\le&\frac {C^qq^{\frac{3q}2}}{n^{\frac q2+1}(np)^{ q\varkappa}}
,\\
\mathcal A_3\le &\frac{C^qq^{2q}}{n^2(np)^{2\varkappa q +2}}.
\end{align*}
Finally we get, for $np\ge C\log n^{\frac1{\varkappa}}$, 
\begin{align*}
\E|Y_{12}|^q\mathbb I\{\mathcal C_1\}\mathbb I\{\mathcal C_2\}\mathbb I\{\mathcal Q\}\le \frac{C^qq^{\frac q2}}{(nv)^{\frac q2}}\Big(a_n^{\frac q2}(z)+\frac{r^{\frac q2}}{(nv)^{\frac q2}}\Big)+\frac{C^qH_2^{\frac q2}s^{\frac q2}q^{\frac{3q}2}}{(nv)^{\frac q2}(np)^{q\varkappa+1}|b(z)|^{\frac q2}}\\+\frac{C^qq^{\frac{3q}2}
	H_1^{\frac q2}s^{\frac q2}}{(nv)^{\frac q2}(np)^{q\varkappa}}+\frac{C^qq^{2q}H_1^qs^q}{(np)^{2\varkappa q+2}}+\frac{C^qH_2^qs^q}{|b(z)|^q(np)^{2\varkappa q+3}}.
\end{align*}
Applying Chebyshev's inequality with $q\sim\log n$, we conclude that
\begin{align}\label{Y_{12}}
\Pr\Big\{|Y_{12}|\ge C\log n\Big(\frac{a_n(z)}{\sqrt{nv}}+\frac{\log^{\frac32}n}{\sqrt{nv}(np)^{\varkappa}|b(z)|^{\frac 12}}+\frac{\log^2n}{(np)^{2\varkappa}|b(z)|}\Big)\Big\}\le Cn^{-c\log\log n}.
\end{align}
Now consider the diagonal elements.
\begin{equation*}
Y_{11}= \sum_{l,t}a_{1l}a_{1t}[\mathbf R^{(\mathbb J)}_{22}]_{lt}.
\end{equation*}
Represent $Y_{11}$ as
\begin{equation*} 
Y_{11}=\sum_{l}a_{1l}^2[\mathbf R^{(\mathbb J)}_{22}]_{ll}+\sum_{l\ne t}a_{1l}a_{1t}[\mathbf R^{(\mathbb J)}_{22}]_{lt}=:\widehat Y_{11}+\widetilde Y_{11}.
\end{equation*}
Applying the inequality for quadratic forms, obtain
\begin{align*}
\E|\widetilde Y_{11}|^q\mathbb I\{\mathcal C_1(sv,1)\}\mathbb I\{\mathcal C_2(sv,1)\}\mathbb I\{\mathcal Q\}&\le C^q\Big(q^q(\E|a_{11}|^2)^{q}\E\|\mathbf R^{(\mathbb J)}\|^q\mathbb I\{\mathcal C_2(sv,1)\}\mathbb I\{\mathcal Q\}\\&+
q^{\frac{3q}2}\mu^{(q)}(\E|a_{11}|^2)^{\frac q2}\sum_l\E\big(\sum_t|[\mathbf R^{(\mathbb J)}_{22}]_{lt}|^2\big)^{\frac q2}\\+q^{2q}(\mu^{(q)})^2\sum_{l,t}&|[\mathbf R^{(\mathbb J)}_{22}]_{lt}|^q\Big)\mathbb I\{\mathcal C_1(sv,1)\}\mathbb I\{\mathcal C_2(sv,1)\}\mathbb I\{\mathcal Q\}.
\end{align*}
From here it is easy to get
\begin{align*}
\E|\widetilde Y_{11}|^q\mathbb I\{\mathcal C_1(sv,1)\}\mathbb I\{\mathcal C_2(sv,1)\}\mathbb I\{\mathcal Q\}&\le C^q\Big(\frac{q^qr^{\frac q2}a_n^{\frac q2}(z)}{(nv)^{\frac q2}}+\frac{q^{\frac{3q}2C^{q}}}{(nv)^{\frac q2}(np)^{\varkappa q+2}|b(z)|^{\frac q2}}\\&\qquad\qquad+\frac{C^qq^{2q}}{(np)^{2\varkappa q+3}|b(z)|^{q}}+\frac{C^q}{(np)^{2\varkappa q+2}}\Big).
\end{align*}
This yields
\begin{align*}
\Pr\Big\{|\widetilde Y_{11}|&\ge C\Big(\frac{\log^2 n\log\log na_n^{\frac12}(z)}{\sqrt{nv}}+\frac{\log^{\frac 52}n}{\sqrt{nv}(np)^{\varkappa}\sqrt{|b(z)|}}\\&+\frac{\log^3 n}{(np)^{2\varkappa}|b(z)|}\Big);\mathcal C_1(sv,1)\cap\mathcal C_2(sv,1)\cap\mathcal Q\Big\}\le Cn^{-\log\log n}.
\end{align*}
Now consider $ \widehat Y_{11}$. We have
\begin{align*}
\widehat Y_{11}=&\frac yn\sum_{l}R^{(\mathbb J)}_{ll}+\sum_{l}(a_{1l}^2-\E a_{1l}^2)R^{(\mathbb J)}_{ll}\\&=yS_y(z)-\frac{1-y}z+y\Lambda_n(z)+\frac{r}{nv}+\sum_{l}(a_{1l}^2-\E a_{1l}^2)R^{(\mathbb J)}_{ll}.
\end{align*}
By Rosenthal's inequality,
\begin{align*}
\E\Big|\sum_{l}(a_{1l}^2-\E a_{1l}^2)R^{(\mathbb J)}_{ll}\Big|^q\mathbb I\{\mathcal C_1(sv,1)\}\mathbb I\{\mathcal C_2(sv,1)\}\mathbb I\{\mathcal Q\}\le &
C^q\Big(\frac{q^{\frac q2}s^q}{(np)^{\frac q2}}
\\+\frac{q^{\frac q2}s^q}{(np)^{q}|b(z)|^q}
+\frac{s^qq^q}{(np)^{2\varkappa q+2}|b(z)|^q}\Big).&
\end{align*}
The obtained bounds give
	\begin{align*}
	\Pr\Big\{&\Big|\widehat Y_{11}-\Big(-\frac{1-y}z+yS_y(z)\Big)\Big|\ge C\Big(\gamma a_n(z)+\frac r{nv}\notag\\&+\frac{\log ^{\frac32}n}{(np)^{\frac12}}+\frac{\log^{\frac32}}{(np)|b(z)|}+\frac{\log^2n}{(np)^{2\varkappa}|b(z)|}\Big);\mathcal C_1\cap\mathcal C_2\cap\mathcal Q\Big\}\le Cn^{-\log\log n}.
	\end{align*}
Summing up the estimates for $\widehat Y_{11}$ and $\widetilde Y_{11}$, we conclude that
\begin{equation*}
\Pr\Big\{\Big|Y_{11}-\Big(yS_y(z)-\frac{1-y}z\Big)\Big|\ge\mathcal  G_1+\mathcal G_2;\mathcal C_1\cap\mathcal C_2\cap\mathcal Q\Big\}\le Cn^{-c\log n},
\end{equation*}
where
\begin{align*}
\mathcal G_1=&\gamma a_n(z),\\
\mathcal G_2=&C\Big(\frac r{nv}+\frac{\log ^{\frac32}n}{(np)^{\frac12}}+\frac{\log^{\frac32}n}{(np)|b(z)|}+\frac{\log^2n}{(np)^{2\varkappa}|b(z)|}\\&+\frac{\log^2 n\log\log na_n^{\frac12}(z)}{\sqrt{nv}}+\frac{\log^{\frac 52}n}{\sqrt{nv}(np)^{\varkappa}\sqrt{|b(z)|}}+\frac{\log^3 n}{(np)^{2\varkappa}|b(z)|}\Big).
\end{align*}
It is easy to show that if
\begin{equation*}
|b(z)|\ge C\log n^{\frac 32}n\bigg(\frac1{\sqrt{nv}}+\frac1{(np)^{\varkappa}}\bigg),
\end{equation*}
then
\begin{equation*}
\Pr\Big\{\Big|Y_{11}-\Big(yS_y(z)-\frac{1-y}z\Big)\Big|\le \gamma|b(z)|;\mathcal C_1\cap\mathcal C_2\cap\mathcal Q\Big\}
\le Cn^{-c\log n}
\end{equation*}
with an arbitrarily small constant $\gamma$. From this and the inequality \eqref{Y_{12}} it follows that

\begin{equation*}
\Pr\Big\{\Big\|\mathbf Y-\Big(yS_y(z)-\frac{1-y}z\Big)\mathbf I\Big\|\ge \gamma|b(z)|;\mathcal C_1\cap\mathcal C_2\cap\mathcal Q\Big\}\ge 1-Cn^{-c\log n}.
\end{equation*}
Since the matrix $\mathbf V_{11}$ is Hermitian (the eigenvalues are real), and
\begin{equation*}
\im\Big(z-\frac{1-y}z+yS_y(z)\Big)\ge \frac {\sqrt 2}2|b(z)|,
\end{equation*}
we find that
\begin{equation*}
\Pr\Big\{\Big\|\Big(\mathbf V_{11}-\Big(z-\frac{1-y}z+yS_y(z)\Big)-\mathbf Y\Big)^{-1}\Big\|\le \frac C{|b(z)|};\mathcal C_1\cap\mathcal C_2\cap\mathcal Q\Big\}\ge 1-Cn^{-c\log\log n}.
\end{equation*}
This, in particular, implies that 
\begin{equation*}
\Pr\Big\{|R_{jk}|\le {H_2};\mathcal C_1(sv,1)\cap\mathcal C_2(sv,1)\cap\mathcal Q\Big\}\ge 1-Cn^{-c\log\log n},
\end{equation*}
for $j\in\mathcal D$.
The last statement completes the proof of the Lemma \ref{lem_2.3}.
 \end{proof}
\begin{proof}[Proof of Lemma \ref{typical}]
Let $k=|\mathbb J|$. Lemma \ref{lem_2.3} and inequality $\max_{j,l}|R^{(\mathbb J)}_{jl}(V)|\le V^{-1}$ imply
	\begin{equation*}
	\Pr\Big\{\mathcal C_1(v,k-1);\mathcal Q\Big\}\ge 1-Cn^{-c\log\log n},
	\end{equation*}
	\begin{equation*}
	\Pr\Big\{\mathcal C_2(v,k-1);\mathcal Q\Big\}\ge 1-Cn^{-c\log\log n},
	\end{equation*}
for $V/s_0\le v\le V$. We may repeat this procedure $L(v_0,s_0)$ times and obtain
\begin{equation*}
\Pr\{\max_{1\le j,k\le n+ m}|R_{jk}(v)|>H;\mathcal Q\}\le Cn^{-c\log\log n},
\end{equation*}
for $v\ge V/s_0^L=v_0$.
 \end{proof}

\begin{proof}[Proof of Lemma \ref{lem_main}]
We recall that Lemma \ref{inadm} gives
\begin{equation*}
\Pr\{\mathbf L_{\mathbf V}\notin\mathcal A\}\le Cn^{-c\log\log n}.
\end{equation*}
It implies Lemma \ref{lem_main}.
 \end{proof}

\section{Appendix} \label{aux}
Let $\xi_1,\ldots,\xi_n$ and $\eta_1,\ldots,\eta_n$ be mutually independent random variables,  $A=(a_{ij})_{i,j=1}^n$. Define
$$
\mathcal L_j^2=\sum_{i=1}^n|a_{ij}|^2.
$$
Note that
$$
\|A\|^2=\sum_{j=1}^n\mathcal L_j^2.
$$
\begin{lem}\label{bilin} For any $q\ge 2$
	the inequality
	\begin{equation*}
	\E|\sum_{i,j=1}^na_{ij}\xi_i\eta_j|^q\le C^q\big(\mathcal A_1\|A\|^q+\mathcal A_2(\sum_{j=1}^n\mathcal L_j^q)+\mathcal A_3(\sum_{i,j=1}^n|a_{ij}|^q)\big)
	\end{equation*}
	holds, where
	\begin{align*} \mathcal A_1=&q^{\frac{3q}2}(\sigma_{\xi}^{2q}+\sigma_{\eta}^{2q}),\\ 
 \mathcal A_2=&q^{\frac{3q}2}(\sigma_{\xi}^{(2q)}+\sigma_{\eta}^{2q})^{\frac{q-6}{2(q-4}}(\mu_{\xi}^{(\frac q2)})^{\frac{2(q-2)}{q-4}},\\
 \mathcal A_3=&q^{2q}\mu_{\xi}^{(q)}\mu_{\eta}^{(q)}.\end{align*}
\end{lem}
\begin{proof}
Let $A=\sum_{i,j=1}^na_{ij}\xi_i\eta_j=\sum_{i=1}^n\xi_i(\sum_{j=1}^na_{ij}\eta_j).$
Applying Rosenthal's inequality, we get
\begin{align*}
A\le C^q(q^{\frac q2}\sigma_{\xi}^q\E\Big(\sum_{i=1}^n\big(\sum_{j=1}^n\eta_ja_{ij})^2\big)^{\frac q2}+q^q\mu_{\xi}^{(q)}\sum_{i=1}^n\E\big|\sum_{j=1}^n a_{ij}\eta_j\big|^q\Big)\notag\\=:C^qq^{\frac q2}\sigma_{\xi}^qA_1+C^qq^q\mu_{\xi}^{(q)}A_2.
\end{align*}
Using the triangle inequality, we obtain
 \begin{align*}
 A_{1}\le 2^{\frac q2}\E\Big( \sum_{i=1}^n\sum_{j=1}^na_{ij}^2\eta_j^2\Big)^{\frac q2}+2^{\frac q2}\E\Big(\sum_{i\ne j}\eta_i\eta_j\big(\sum_{l=1}^na_{il}a_{lj}\big)\Big)^{\frac q2}\notag\\=:2^{\frac q2}( A_{11}+ A_{12}).
 \end{align*}
 Further,
 \begin{equation*}
  A_{11}\le 2^{\frac q2}\Big(\big(\sum_{i=1}^n\sum_{j=1}^n a_{ij}^2\big)^{\frac q2}\sigma_{\eta}^q+\E\Big(\sum_{j=1}^n
(\eta_j^2-\sigma_{\eta}^2)\big(\sum_{i=1}^na_{ij}^2 \big)\Big)^{\frac q2}\Big).
 \end{equation*}
Applying Rosenthal's inequality again, we conclude that
 \begin{align*}
 \E\Big(\sum_{j=1}^n
 (\eta_j^2-\sigma_{\eta}^2)\big(\sum_{i=1}^na_{ij}^2\big) \Big)^{\frac q2}&\le C^qq^{\frac q4}\Big(\sum_{j=1}^n\big(\sum_{i=1}^na_{ij}^2\big)^2\Big)^{\frac q4}(\mu_{\eta}^{(4)})^{\frac q4}\notag\\&+C^qq^{\frac q2}\mu_{\eta}^{(q)}\sum_{j=1}^n\big(\sum_{i=1}^n a_{ij}^2\big)^{\frac q2}.
 \end{align*}
 To estimate $A_{12}$, we use the inequality for quadratic forms from \cite{Tikh:2018}. We have
 \begin{align*}
 A_{12}\le &C^qq^{\frac q2}\sigma_{\eta}^q\Big(\sum_{i\ne j}\big(\sum_{l=1}^na_{il}a_{lj}\big)^2\Big)^{\frac q4}+C^qq^{\frac {3q}4}\mu_{\eta}^{(\frac q2)}\sigma_{\eta}^{\frac q2}\sum_{j=1}^n\Big(\sum_{i=1}^n (\sum_{l=1}^na_{il}a_{lj})^{2}\Big)^{\frac q4}\notag\\&+C^q\Big(q^{q}(\mu_{\eta}^{(\frac q2)})^2\Big(\sum_{i\ne j}\big(\sum_{l=1}^n a_{il}a_{lj}\big)^{\frac q2}\Big).
 \end{align*}
Summing up the above inequalities, we find that
 \begin{align*}
 A_1\le& C^q\sigma_{\eta}^q\Big(\sum_{i=1}^n\sum_{j=1}^na_{ij}^2\Big)^{\frac q2}+C^q(\mu_{\eta}^{(4)})^{\frac q4}q^{\frac q4}\Big(\sum_{j=1}^n\big(\sum_{i=1}^na_{ij}^2\big)^2\Big)^{\frac q4}\notag\\&+C^q\Big(q^{\frac q2}\sigma_{\eta}^q\Big(\sum_{i\ne j}\big(\sum_{l=1}^n a_{il}a_{lj}\big)^{2}\Big)^{\frac q4}+C^qq^{\frac{3q}4}\mu_{\eta}^{(\frac q2)}\sigma_{\eta}^{\frac q2}\sum_{i=1}^n\Big(\sum_{j\ne i}\big(\sum_{l=1}^na_{il}a_{lj}\big)^2\Big)^{\frac q4}\notag\\&+C^qq^{q}(\mu_{\eta}^{(\frac q2)})^2\sum_{i\ne j}\big(\sum_{l=1}^na_{il}a_{lj}\big)^{\frac q2}\Big)+C^qq^{\frac q2}\mu_{\eta}^{(q)}\sum_{j=1}^n\big(\sum_{i=1}^n a_{ij}^2\big)^{\frac q2}.
 \end{align*}
For $A_2$, by Rosenthal's inequality, we have
 \begin{align}
 A_2\le C^q\sigma_{\eta}^qq^{\frac q2}\sum_{i=1}^n\mathcal  L_i^q+C^qq^q\mu_{\eta}^{(q)}\sum_{i=1}^n|a_{ij}|^q.
 \end{align}
 Further note that
 \begin{align*}
 \Big(\sum_{j=1}^n\big(\sum_{i=1}^na_{ij}^2\big)^2\Big)^{\frac q4}&\le \Big(\sum_{i=1}^n\mathcal L_i^{q}\Big)^{\frac{q}{2(q-2)}}(\|A\|^q)^{\frac {q-4}{2(q-2)}},\notag\\
\Big(\sum_{i\ne j}\big(\sum_{l=1}^n a_{il}a_{lj}\big)^{2}\Big)^{\frac q4}&\le\|A\|^q,\\
\sum_{i=1}^n\Big(\sum_{j\ne i}\big(\sum_{l=1}^na_{il}a_{lj}\big)^2\Big)^{\frac q4}&\le \Big(\sum_{j=1}^n\mathcal L_j^q\Big)^{\frac{(q-4)}{2(q-2)}}\big(\|A\|^{q}\big)^{\frac{q}{2(q-2)}},\\
\Big(\sum_{i}\mathcal L_j^{\frac q2}\Big)^2&\le\Big(\sum_{j=1}^n\mathcal L_j^q\Big)^{\frac{(q-4)}{(q-2)}}\big(\|A\|^{q}\big)^{\frac{2}{(q-2)}}.
 \end{align*}
 For $A$ we get the estimate
 \begin{equation*}
 A\le C^q(B_1+\ldots+B_8),
 \end{equation*}
 where
 \begin{align*}
 B_1=&q^{\frac q2}\sigma_{\xi}^q\sigma_{\eta}^q\|A\|^q,\\
B_2=&\sigma_{\xi}^q(\mu_{\eta}^{(4)})^{\frac q4}q^{\frac {3q}4}\Big(\sum_{j=1}^n\mathcal L_j^q\Big)^{\frac q{2(q-2)}}\big(\|A\|^q)^{\frac{q-4}{2(q-2)}},\\
 B_3=&q^q\sigma_{\xi}^q\sigma_{\eta}^q
 \Big(\sum_{j=1}^n\mathcal L_j^q\Big)^{\frac{q-4}{q-2}}\Big(\|A\|^q\Big)^{\frac 2{q-2}},\\
	B_4=&q^{\frac {5q}4}\sigma_{\xi}^q\mu_{\eta}^{(\frac q2)}\sigma_{\eta}^{\frac q2}
	\Big(\sum_{j=1}^n\mathcal L_j^q\Big)^{\frac{q-4}{2(q-2)}}
	\Big(\|A\|^q\Big)^{\frac q {2(q-2)}},\\
 		B_5=&q^{\frac{3q}2}\sigma_{\xi}^q\big(\mu_{\eta}^{(\frac q2)}\big)^2\Big(\sum_{j=1}^n\mathcal L_j^q\Big)^{\frac{q-4}{(q-2)}}\Big(\|A\|^q\Big)^{\frac 2 {(q-2)}},\\
 			B_6=&q^{q}\sigma_{\xi}^q\mu_{\eta}^{(q)}\sum_{j=1}^n\mathcal L_j^q,\\
B_7=&q^{\frac{3q}2}\sigma_{\eta}^q\mu_{\xi}^{(q)}\sum_{j=1}^n\mathcal L_j^q,\\
B_8=&q^{2q}\mu_{\xi}^{(q)}\mu_{\eta}^{(q)}\sum_{i,j=1}^n|a_{ij}|^q.
\end{align*}
 Applying Young's inequality, we obtain the bounds
 \begin{align*}
 B_2\le Cq^{\frac{3q}4}\Big(\sigma_{\xi}^4(\mu_{\eta}^{(4)})^{\frac{q-2}2}\sum_{j=1}^n\mathcal L_j^q+\sigma_{\xi}^{2q}\|A\|^q),\\
 B_3\le C^qq^q\sigma_{\xi}^q\sigma_{\eta}^q\sum_{j=1}^n\mathcal L_j^q+C^qq^q\sigma_{\xi}^q\sigma_{\eta}^q\|A\|^q\Big),\\
 B_4\le C^qq^{\frac{5q}4}\Big((\sigma_{\xi}^{2q}+\sigma_{\eta}^{2q})\|A\|^q
 +(\mu_{\eta}^{(\frac q2)})^{\frac{2(q-2)}{q-4}}(\sigma_{\xi}^{q\frac{q-6}{q-4}}+\sigma_{\eta}^{q\frac{q-6}{q-4}})\sum_{j=1}^n\mathcal L_j^q\Big),\\
 B_5\le C^qq^{\frac{3q}2}\Big(\sigma_{\xi}^{2q}\|A\|^q+\sigma_{\xi}^{\frac{q(q-6)}{q-4}}(\mu_{\eta}^{(\frac q2)})^{\frac{2(q-2)}{q-4}}\sum_{j=1}^n\mathcal L_j^q\Big).
 \end{align*}
 The last inequalities give
 \begin{equation*}
 A\le C^q\big(\mathcal A_1\|A\|^q+\mathcal A_2(\sum_{j=1}^n\mathcal L_j^q)+\mathcal A_3(\sum_{i,j=1}^n|a_{ij}|^q)\big),
 \end{equation*}
 where
 \begin{align*}
 \mathcal A_1=&q^{\frac{3q}2}(\sigma_{\xi}^{2q}+\sigma_{\eta}^{2q}),\\ 
 \mathcal A_2=&q^{\frac{3q}2}(\sigma_{\xi}^{(2q)}+\sigma_{\eta}^{2q})^{\frac{q-6}{2(q-4}}(\mu_{\xi}^{(\frac q2)})^{\frac{2(q-2)}{q-4}},\\
 \mathcal A_3=&q^{2q}\mu_{\xi}^{(q)}\mu_{\eta}^{(q)}.
 \end{align*}
 Thus Lemma is proved.
\end{proof}

%

\end{document}